\documentclass{article}

\usepackage{amsmath}
\usepackage{amsfonts}
\usepackage{amssymb}
\usepackage{amsthm}
\usepackage{cite}
\usepackage{hyperref}

\newtheorem{thm}{Theorem}[section]
\newtheorem{cor}[thm]{Corollary}
\newtheorem{lem}[thm]{Lemma}

\begin{document}
\title{On the Form of Odd Perfect Gaussian Integers}
\author{Matthew Ward}
\maketitle

\footnotetext{Research for this paper was completed at the Summer 2007
Research Experience for Undergraduates at Auburn University, under the
auspices of NSF grant no. 03533723.}

\section{Introduction}

Let $\mathbb{Z}[i]=\{a+bi : a,b\in\mathbb{Z}\}$ be the ring of Gaussian integers. All Gaussian integers will be represented by Greek letters and rational integers by ordinary Latin letters. Primes will be denoted by $\pi$ and $p$ respectively. Units will be denoted by $\varepsilon=\pm 1, \pm i$ and $1$ respectively. In $1961$ Robert Spira \cite{spira} defined the sum-of-divisors function on $\mathbb{Z}[i]$ as follows. Let $\eta=\varepsilon \Pi \pi_i^{k_i}$ be a Gaussian integer. This representation is unique in that we will choose our unit $\varepsilon$ such that each $\pi_i$ is in the first quadrant ($\text{Re} (\pi_i)> 0$ and $\text{Im} (\pi_i)\geq 0$). Spira defined the sum-of-divisors function $\sigma$ as

\[
\sigma(\eta)=\prod \frac{\pi_i^{k_i+1}-1}{\pi_i-1}
\]

This can be seen to match precisely with the rational form of the sum-of-divisors exactly when the rational primes coincide with the gaussian primes (when $p\equiv 3 \mod 4$).

A Gaussian integer $\eta$ is considered \textit{even} if and only if $1+i$ divides $\eta$. It is easy to see that $1 + i$ divides $a + bi$ if and only if $a$ and $b$
have the same parity, i.e., $a \equiv b \mod 2$.  It follows that the usual
parity rules for addition for $\mathbb{Z}$ hold in $\mathbb{Z}[i]$:  The sum of two Gaussian integers of the same parity is even, and of opposite parity, odd.  Also, since $a \equiv a^2 \mod 2$ for any rational integer a, it follows that a Gaussian integer $a + bi$ is even if and only if its norm, $N(a + bi) = a^2 + b^2$, is an even rational integer.

Finally, we can define \textit{perfect numbers} in the natural way. A Gaussian integer $\eta$ is \textit{perfect} if $\sigma (\eta) = (1+i)\eta$. This is not the only notion we can work with, though. The Gaussian integer $\eta$ is \textit{norm-perfect} if $N\left( \sigma(\eta) \right)=N(1+i)N\left(\eta\right)=2 N\left(\eta\right)$. Every perfect number is norm-perfect, so often it is easier to work with the norm perfect concept. Wayne McDaniel \cite{mcdaniel} proved a theorem for Gaussian integers analogous to Euclid's and Euler's characterization of even perfect positive integers. We will prove a theorem for Gaussian integers analogous to Euler's partial characterization of odd perfect positive integers, which is that any such integer must be of the form $p^jm^2$ where $\gcd(p,m) = 1$ and $p \equiv j \equiv 1 \mod 4$.

\section{The Form of Odd Perfect Numbers}

\begin{lem}\label{lemma}
If $\pi$ is an odd prime, then $\sigma (\pi^m)$ is even if and only if $m$ is odd.
\end{lem}

\begin{proof}
We first claim that for any Gaussian integers $\eta_j$ we have $\displaystyle N\left(\sum_j \eta_j \right) \equiv \sum_j N(\eta_j) \mod 2$. This follows from basic parity rules. We have $\displaystyle N\left(\sum_j \eta_j \right) \equiv 0 \mod 2$ if and only if $\displaystyle \sum_j \eta_j$ is even in the Gaussian sense, but this happens if and only if there are an even number of odd $\eta_j$'s. This means that there are an even number of $N(\eta_j)$ $\equiv 1\mod 2$. Therefore, we have that $N\left( \displaystyle \sum_j \eta_j \right) \equiv 0 \mod 2$ if and only if $\displaystyle \sum_j N(\eta_j )\equiv 0\mod 2$.

Suppose $\pi$ is an odd prime. Now examine $\displaystyle N\left(\sigma(\pi^m) \right)=N\left(1+\pi + \cdots + \pi^m \right) \equiv N(1) + N(\pi) + \cdots + N(\pi^m) \mod 2$. We clearly have that this sum is congruent to $0 \mod 2$ if and only if $m$ is odd since each term in the sum is odd.
\end{proof}

\begin{thm}
If $\alpha$ is an odd norm-perfect Gaussian integer, then $\alpha = \pi^k \gamma^2$ where $k$ is an odd rational integer and $\gcd(\pi, \gamma)=1$.
\end{thm}

\begin{proof}
Let $\alpha$ be an odd norm-perfect Gaussian integer. Then $\displaystyle\alpha=\prod_{i=1}^n \pi_i^{k_i}$, where no $\pi_j$ is associate to $(1+i)$. Since $\alpha$ is norm-perfect we have $N\left( \sigma(\alpha)\right) = 2 N(\alpha )$, but since $N(\alpha)$ is odd and the sum of two squares we have that $N(\alpha)\equiv 1\mod 4$, so $N\left(\sigma(\alpha)\right)\equiv 2 N(\alpha) \equiv 2\mod 4$.

Examine \begin{eqnarray*}
N\left( \sigma(\alpha) \right) & = & \prod_{i=1}^n N\left(\sigma\left(\pi_i^{k_i}\right)\right) \\
										& = & N\left(\sigma\left(\pi_1^{k_1}\right)\right) \cdots N\left(\sigma\left(\pi_n^{k_n}\right)\right) \\
\end{eqnarray*}

Without loss of generality we can suppose that $N\left(\sigma\left(\pi_1^{k_1}\right)\right) \equiv 2\mod 4$ and all the other terms in the product above are congruent to $1 \mod 4$. This is because no term can be $0\mod 4$ or else the whole product would be $0\mod 4$, and no term can be $3\mod 4$ since the norm is the sum of two squares. From Lemma \ref{lemma} we know that $k_1$ is odd and that each $k_i$ is even for all $1<i\leq n$. 

The form of an odd norm-perfect Gaussian integer then must have one prime to an odd power and the rest of the factorization are squares. Thus the form is $\alpha=\pi^k \gamma^2$ where $k$ is odd and $\gcd(\pi, \gamma)=1$.
\end{proof}

\begin{cor}
If $\alpha$ is an odd perfect Gaussian integer, then $\alpha = \pi^k \gamma^2$ where $k$ is an odd rational integer and $\gcd(\pi, \gamma)=1$.
\end{cor}

\begin{proof}
Every perfect Gaussian integer is norm-perfect.
\end{proof}

\section{Some Pesky Counterexamples}

As noted above, in the rational case, Euler proved that every odd perfect number has the form $p^jb^2$ where $p$ is a prime, $\gcd(p,b)=1$, $p\equiv 1\mod 4$, and $j\equiv 1\mod 4$. In our form, $\pi_1^k\gamma^2$,  we get that $N( \pi_1) \equiv 1\mod 4$ for free since the norm is a sum of two squares. However, we cannot show that $k \equiv 1\mod 4$. 

It is still unknown whether or not there are any odd perfect Gaussian integers, but there are odd norm-perfect numbers. In fact the smallest is a rather disturbing example: $2+i$. It turns out that this is not only odd, but a prime. Have no fear, we shall show that this is the only such example. Let $\pi$ be a Gaussian prime. If $\pi$ is norm-perfect then $N(\pi + 1)= 2N(\pi)$; if $\pi=a+bi$ then this equation becomes

\[
(a+1)^2+b^2=2(a^2+b^2)
\]
which is equivalent to $(a-1)^2+b^2=2$. The only $4$ integer solutions to this are $(2,-1)$, $(0,-1)$, $(2,1)$, $(0,1)$. Therefore in the Gaussian integers $2+ i$ and $2 - i$ are the only primes that are norm-perfect.

\section{Acknowledgments}

I would like to thank the NSF and Auburn University for providing the REU and facilities for this research. I would like to acknowledge Pete Johnson of Auburn University for advising me and providing comments on this paper. I also acknowledge Jeff Ward for introducing me to the sum-of-divisors and allowing me to bounce ideas off of him all summer.

\nocite{muskat}
\bibliography{oddperfect}
\bibliographystyle{plain}

\end{document}